\documentclass[10pt,leqno]{amsart}
\usepackage{indentfirst, amssymb,amsmath,amsthm}
\theoremstyle{definition}
\newtheorem*{theoA}{Theorem A}
\newtheorem*{theoB}{Theorem B}
\newtheorem*{theoC}{Theorem C}
\newtheorem*{theoD}{Theorem D}
\newtheorem*{theoE}{Theorem E}

\newtheorem{theo}{Theorem}[section]
\newtheorem{lem}{Lemma}[section]

\newtheorem{defi}{Definition}[section]
\newtheorem{rem}{Remark}[section]
\newtheorem{ques}{Question}[section]
\newcommand{\ol}{\overline}
\newcommand{\be}{\begin{equation}}
\newcommand{\ee}{\end{equation}}
\newcommand{\beas}{\begin{eqnarray*}}
\newcommand{\eeas}{\end{eqnarray*}}
\newcommand{\bea}{\begin{eqnarray}}
\newcommand{\eea}{\end{eqnarray}}
\newcommand{\lra}{\longrightarrow}
\numberwithin{equation}{section}
\begin{document}
\title[Some uniqueness results related to  the Br\"{u}ck Conjecture]{Some uniqueness results related to  the Br\"{u}ck Conjecture}
\date{}
\author[B. Chakraborty ]{Bikash Chakraborty }
\date{}
\address{Department of Mathematics, Ramakrishna Mission Vivekananda Centenary College, West Bengal 700 118, India.}
\email{bikashchakraborty.math@yahoo.com, bikash@rkmvccrahara.org}
\maketitle
\let\thefootnote\relax
\begin{abstract} Let $f$ be a non-constant meromorphic function and $a=a(z)(\not\equiv 0, \infty)$ be a small function of $f$. Under
certain essential conditions, we obtained similar type conclusion of Br\"{u}ck Conjecture, when $f$ and its differential polynomial $P[f]$ shares $a$ with weight $l(\geq0)$. Our result improves and generalizes a recent result of Li, Yang and Liu.\end{abstract}
\section{Introduction}
Let $f$ and $g$  be two non-constant meromorphic functions in the open complex plane $\mathbb{C}$. If for some $a\in\mathbb{C}\cup\{\infty\}$, $f$ and $g$ have same set of $a$-points with the same multiplicities, we say that $f$ and $g$ share the value $a$ CM (counting multiplicities) and if we do not consider the multiplicities, then $f$ and $g$ are said to share the value $a$ IM (ignoring multiplicities).  When $a=\infty$, the zeros of $f-a$ means the poles of $f$.\par
A meromorphic function $a(z)(\not\equiv 0, \infty)$ is called a small function with respect to $f$ provided that $T(r,a)=S(r,f)$ as $r\lra \infty, r\not\in E$, where $E$ is any set of positive real numbers whose Lebesgue measure is finite.\par
If $a=a(z)$ is a small function, then we say that $f$ and $g$ share $a$ IM or $a$ CM according as $f-a$ and $g-a$ share $0$ IM or $0$ CM respectively.\par
The hyper order $\rho _{2}(f)$ of a non-constant meromorphic function $f$ is defined by \beas \rho _{2}(f)=\limsup\limits_{r\lra\infty}\frac{ \log \log T(r,f)}{\log r}.\eeas
In connection to find the relation between an entire function with its derivative when they share one value CM, in 1996, in this direction the following famous conjecture was proposed by Br\"{u}ck \cite{3}:\\ \par
{\it {\bf Conjecture:} Let $f$ be a non-constant entire function such that the hyper order $\rho _{2}(f)$ of $f$ is not a positive integer or infinite. If $f$ and $f^{'}$ share a finite value $a$ CM, then $\frac{f^{'}-a}{f-a}=c$, where $c$ is a non-zero constant.}\\ \par
Br\"{u}ck himself proved the conjecture for $a=0$ and for $a=1$, he showed that under the assumption $N(r,0;f^{'})=S(r,f)$ the conjecture was true.
\begin {theoA}\cite{3} Let $f$ be a non-constant entire function. If $f$ and $f^{'}$ share the value $1$ CM and if $N(r,0;f^{'})=S(r,f)$, then $\frac{f^{'}-1}{f-1}$ is a nonzero constant.
\end {theoA}
However for entire function of finite order, Yang \cite{11} removed the supposition $N(r,0;f^{'})=0$ and obtained the following result.
\begin {theoB}\cite{11} Let $f$ be a non-constant entire function of finite order and let $a(\not=0)$ be a finite constant. If $f$, $f^{(k)}$ share the value $a$ CM, then $\frac{f^{(k)}-a}{f-a}$ is a nonzero constant, where $k(\geq 1)$ is an integer.
\end{theoB}
Zhang \cite{14} extended {\it Theorem A} to meromorphic function and also studied the CM value sharing of a meromorphic function with its $k$-th derivative.\par
Meanwhile a new notion of scalings between CM and IM known as weighted sharing (\cite{6}), appeared in the uniqueness literature.
\begin{defi} \cite{6} Let $k$ be a non-negative integer or infinity. For $a\in\mathbb{C}\cup\{\infty\}$, we denote by $E_{k}(a;f)$, the set of all $a$-points of $f$, where an $a$-point of multiplicity $m$ is counted $m$ times if $m\leq k$ and $k+1$ times if $m>k$. If $E_{k}(a;f)=E_{k}(a;g)$, we say that $f$ and $g$ share the value $a$ with weight $k$.
\end{defi}
The definition implies that if $f$ and $g$ share a value $a$ with weight $k$, then $z_{0}$ is an $a$-point of $f$ with multiplicity $m\;(\leq k)$ if and only if it is an $a$-point of $g$ with multiplicity $m\;(\leq k)$ and $z_{0}$ is an $a$-point of $f$ with multiplicity $m\;(>k)$ if and only if it is an $a$-point of $g$ with multiplicity $n\;(>k)$, where $m$ is not necessarily equal to $n$.\par

We write $f$, $g$ share $(a,k)$ to mean that $f$, $g$ share the value $a$ with weight $k$. Clearly if $f$, $g$ share $(a,k)$, then $f$, $g$ share $(a,p)$ for any integer $p$, $0\leq p<k$. Also we note that $f$, $g$ share a value $a$ IM or CM if and only if $f$, $g$ share $(a,0)$ or $(a,\infty)$ respectively.\par
Though out this paper, we use the standard notations and definitions of the value distribution theory available in \cite{4}. Also we explain some definitions and notations which are used in this paper.
\begin{defi}\cite{8} Let $p$ be a positive integer and $a\in\mathbb{C}\cup\{\infty\}$.
\begin{enumerate}
\item[(i)] $N(r,a;f\mid \geq p)$ ($\ol N(r,a;f\mid \geq p)$) denotes the counting function (reduced counting function) of those $a$-points of $f$ whose multiplicities are not less than $p$.\item[(ii)]$N(r,a;f\mid \leq p)$ ($\ol N(r,a;f\mid \leq p)$) denotes the counting function (reduced counting function) of those $a$-points of $f$ whose multiplicities are not greater than $p$.
\end{enumerate}
\end{defi}
\begin{defi}\{6, cf.\cite {12}\} For $a\in\mathbb{C}\cup\{\infty\}$ and a positive integer $p$ we denote by $N_{p}(r,a;f)$ the sum $\ol N(r,a;f)+\ol N(r,a;f\mid\geq 2)+\ldots\ol +N(r,a;f\mid\geq p)$. Clearly $N_{1}(r,a;f)=\ol N(r,a;f)$. \end{defi}
\begin{defi} Let $k$ be a positive integer and for $a\in\mathbb{C}\setminus\{0\}$, $\ol E_{k)}(a;f)=\ol E_{k)}(a;g)$. Let $z_{0}$ be a zero of $f(z)-a$ of multiplicity $p$ and a zero of $g(z)-a$ of multiplicity $q$.
\begin{enumerate}
\item[(i)] We denote by $\ol N_{L}(r,a;f)$ the counting function of those $a$-points of $f$ and $g$ where $p>q\geq 1$,
\item[(ii)] by $\ol N_{f>s}(r,a;g)$ (resp. $\ol N_{g>s}(r,a;f)$) the counting functions of those $a$-points of $f$ and $g$ for which $p>q=s$ (resp. $q>p=s$),
\item[(iii)] by $N^{1)}_{E}(r,a;f)$ the counting function of those $a$-points of $f$ and $g$ where $p=q=1$ and
\item[(iv)] by $\ol N^{(2}_{E}(r,a;f)$ the counting function of those $a$-points of $f$ and $g$ where $p=q\geq 2$, each point in these counting functions is counted only once.
\end{enumerate}
In the same way, we can define $\ol N_{L}(r,a;g),\; N^{1)}_{E}(r,a;g),\; \ol N^{(2}_{E}(r,a;g).$ We denote by $\ol N_{f\geq k+1}(r,a;f\mid\; g\not=a)$ (resp. $\ol N_{g\geq k+1}(r,a;g\mid\; f\not=a)$) the reduced counting functions of those $a$-points of $f$ and $g$ for which $p\geq k+1$ and $q=0$ (resp. $q\geq k+1$ and $p=0$).
\end{defi}
\begin{defi}\cite{7} Let $a,b \in\mathbb{C}\;\cup\{\infty\}$. We denote by $N(r,a;f\mid\; g\neq b)$ the counting function of those $a$-points of $f$, counted according to multiplicity, which are not the $b$-points of $g$.\end{defi}
\begin{defi}\cite{6} Let $f$, $g$ share a value $a$ IM. We denote by $\ol N_{*}(r,a;f,g)$ the reduced counting function of those $a$-points of $f$ whose multiplicities differ from the multiplicities of the corresponding $a$-points of $g$.

Clearly $\ol N_{*}(r,a;f,g)\equiv\ol N_{*}(r,a;g,f)$ and $\ol N_{*}(r,a;f,g)=\ol N_{L}(r,a;f)+\ol N_{L}(r,a;g)$.
\end{defi}

\begin{defi} (\cite{10}) For $a \in \mathbb{C}\cup\{\infty\}$ and a positive integer p, we put\\
$$\delta_{p}(a,f)=1-\limsup\limits_{r \to \infty} \frac{{N}_{p}(r,a;f)}{T(r,f)}.$$\\ Clearly $0\leq \delta(a,f) \leq \delta_{p}(a,f) \leq \delta_{p-1}(a,f) \leq...\leq \delta_{2}(a,f) \leq \delta_{1}(a,f)=\Theta(a,f)\leq1$ .
\end{defi}
In 2004, Lahiri-Sarkar \cite{8} employed weighted value sharing method to improve the results of Zhang \cite{14}.
In 2005, Zhang \cite{15} further extended the results of Lahiri-Sarkar to a small function and proved the following result for IM sharing.
\begin{theoC}\cite{15} Let $f$ be a non-constant meromorphic function and $k(\geq 1)$ and $l(\geq 0)$ be integer. Also let $a\equiv a(z)$ ($\not\equiv 0,\infty$) be a meromorphic small function. Suppose that $f-a$ and $f^{(k)}-a$ share $(0,l)$. If $l\geq2$ and
\be\label{e1.1} 2\ol N(r,\infty;f)+N_{2}\left(r,0;f^{(k)}\right)+ N_{2}\left(r,0;(f/a)^{'}\right)<(\lambda +o(1))\;T\left(r,f^{(k)}\right)\ee
or, $l=1$ and
\be\label{e1.1} 2\ol N(r,\infty;f)+N_{2}\left(r,0;f^{(k)}\right)+ 2\overline{N}\left(r,0;(f/a)^{'}\right)<(\lambda +o(1))\;T\left(r,f^{(k)}\right)\ee
or, $l=0$ and
\be\label{e1.1} 4\ol N(r,\infty;f)+3N_{2}\left(r,0;f^{(k)}\right)+ 2\overline{N}\left(r,0;(f/a)^{'}\right)<(\lambda +o(1))\;T\left(r,f^{(k)}\right)\ee
for $r\in I$, where $0<\lambda <1$ and $I$ is a set of infinite linear measure, then $\frac{f^{(k)}-a}{f-a}=c$  for some constant $c\in\mathbb{C}/\{0\}$.
\end{theoC}

Let $a_{j}~(j = 0, 1,\ldots , k - 1)$ are small meromorphic functions of $f$. We define
$$L(f)=f^{(k)}+a_{k-1}f^{(k-1)} + \ldots+ a_{0}f.$$
In 2007, Zhang and Yang (\cite{zy}) obtained the following result:
\begin{theoD}\cite{zy} Let $f$ be a non-constant meromorphic function and $k(\geq 1)$ and $l(\geq 0)$ be integer. Also let $a\equiv a(z)$ ($\not\equiv 0,\infty$) be a meromorphic small function. Suppose that $f-a$ and $L(f)-a$ share $(0,l)$. If $l\geq2$ and
\be\label{e1.1} \delta_{2+k}(0,f)+ \delta_{2}(0,f)+3\theta(\infty,f)+ \delta(a,f)>4,\ee
or, $l=1$ and
\be\label{e1.1} \delta_{2+k}(0,f)+ \delta_{2}(0,f)+\frac{1}{2}\delta_{1+k}(0,f)+\frac{k+7}{2}\theta(\infty,f)+ \delta(a,f)>\frac{k}{2}+5,\ee
or, $l=0$ and
\be\label{e1.1} \delta_{2+k}(0,f)+2\delta_{1+k}(0,f)+ \delta_{2}(0,f)+\Theta(0,f)+(6+2k)\theta(\infty,f)+ \delta(a,f)>2k+10,\ee
then $f=L(f)$  for some constant $c\in\mathbb{C}/\{0\}$.
\end{theoD}
\begin{defi} Let $n_{0j},n_{1j},\ldots,n_{kj}$ be non-negative integers. The expression $$M_{j}[f]=(f)^{n_{0j}}(f^{(1)})^{n_{1j}}\ldots(f^{(k)})^{n_{kj}}$$ is called a differential monomial generated by $f$ of degree $d(M_{j})=\sum\limits_{i=0}^{k}n_{ij}$ and weight
$\Gamma_{M_{j}}=\sum\limits_{i=0}^{k}(i+1)n_{ij}$. The sum $$P[f]=\sum\limits_{j=1}^{t}b_{j}M_{j}[f]$$ is called a differential polynomial generated by $f$ of degree $d=\ol{d}(P)=max\{d(M_{j}):1\leq j\leq t\}$
and weight $\Gamma=\Gamma_{P}=max\{\Gamma_{M_{j}}:1\leq j\leq t\}$, where $T(r,b_{j})=S(r,f)$ for $j=1,2,\ldots,t$.\par
The numbers $\underline{d}(P)=min\{d(M_{j}):1\leq j\leq t\}$ and $k$ (the highest order of the derivative of $f$ in $P[f]$) are called respectively the lower degree and \emph{order} of $P[f]$.\par
$P[f]$ is said to be \emph{homogeneous} if $\ol{d}(P)$=$\underline{d}(P)$. Otherwise, $P[f]$ is called non-homogeneous differential polynomial.\par
We denote by $\mu =max\; \{\Gamma _{M_{j}}-d(M_{j}): 1\leq j\leq t\}=max\; \{ n_{1j}+2n_{2j}+\ldots+kn_{kj}: 1\leq j\leq t\}$.
\end{defi}

Recently Li, Yang and Liu (\cite{ly}) improved the above Theorems and obtained the following result:

\begin{theoE}\cite{ly} Let $f$ be a non-constant meromorphic function and $P[f]$ be a non-constant homogeneous differential polynomial of degree $d$ and weight $\Gamma_{P}$ satisfying $\Gamma_{P}\geq (k+2)d-2$. Let  $l(\geq 0)$ be integer. Also let $a\equiv a(z)$ ($\not\equiv 0,\infty$) be a meromorphic small function. Suppose that $f-a$ and $P[f]-a$ share $(0,l)$. If $l\geq2$ and
\be\label{e1.1} d\delta_{2+\Gamma_{P}-d}(0,f^{d})+ \delta_{2}(0,f)+3\theta(\infty,f)+ \delta(a,f)>4\ee
or, $l=1$ and
\be\label{e1.1} d\delta_{2+\Gamma_{P}-d}(0,f^{d})+ \delta_{2}(0,f)+\frac{d}{2}\delta_{1+\Gamma_{P}-d}(0,f^{d})+\frac{7+\Gamma_{P}-d}{2}\theta(\infty,f)+ \delta(a,f)>\frac{\Gamma_{P}+9}{2}\ee
or, $l=0$ and
\be\label{e1.1} d\delta_{2+\Gamma_{P}-d}(0,f^{d})+2d\delta_{1+\Gamma_{P}-d}(0,f^{d})+ \delta_{2}(0,f)+\theta(0,f)+[6+2(\Gamma_{P}-d)]\theta(\infty,f)+ \delta(a,f)>2\Gamma_{P}+8\ee
then $\frac{P[f]-a}{f-a}=c$  for some constant $c\in\mathbb{C}/\{0\}$.\par
Especially, when $l=0$ and (\ref{e1.1}) satisfied, then $P[f]=f$.
\end{theoE}

\begin{ques} Can Br\"{u}ck type conclusion be obtained when homogeneous differential polynomial is replaced by arbitrary differential polynomial in Theorem E ?
\end{ques}

\begin{theo}\label{t1}  Let $f$ be a non-constant meromorphic function and $P[f]$ be a non-constant differential polynomial of degree $\ol{d}(P)$ and weight $\Gamma$ satisfying $ \Gamma> (k+1) \underline{d}(P)-2$. Let  $l(\geq 0)$ be integer. Also let $a\equiv a(z)$ ($\not\equiv 0,\infty$) be a meromorphic small function. Suppose that $f-a$ and $P[f]-a$ share $(0,l)$. If $l\geq2$ and
\be\label{be1.1} 3\Theta(\infty,f)+\delta_{2}(0,f)+\underline{d}(P)\delta_{_{2+\Gamma-\underline{d}(P)}}(0;f)+\delta(a,f)
> 4\ee
or, $l=1$, $2\underline{d}(P)>\overline{d}(P)$ and
\bea\label{be1.2} &&\frac{7+\Gamma-\underline{d}(P)}{2}\Theta(\infty,f)+\frac{\underline{d}(P)}{2}\delta_{_{1+\Gamma-\underline{d}(P)}}(r,0;f^{\underline{d}(P)})+
\underline{d}(P)\delta_{_{2+\Gamma-\underline{d}(P)}}(r,0;f^{\underline{d}(P)})\\
\nonumber&&\hspace{4.7 cm}+\delta_{2}(0,f)+\delta(a,f)> \frac{9+\Gamma}{2}+\overline{d}(P)-\underline{d}(P),\eea
or, $l=0$, $5\underline{d}(P)>4\overline{d}(P)$ and
\bea\label{be1.3}&&2(\Gamma-\underline{d}(P)+3)\Theta(\infty,f)+\Theta(0,f)+\delta_{2}(0,f)+\underline{d}(P)\delta_{_{2+\Gamma-\underline{d}(P)}}(r,0;f^{\underline{d}(P)})\\
\nonumber&& +2\underline{d}(P)\delta_{_{1+\Gamma-\underline{d}(P)}}(r,0;f^{\underline{d}(P)})+\delta(a,f)> 2(\Gamma+4)+4(\overline{d}(P)-\underline{d}(P))\eea\\
then $\frac{P[f]-a}{f-a}=c$  for some constant $c\in\mathbb{C}/\{0\}$.
\end{theo}
\begin{rem} If $P[f]$ be a non-constant homogeneous differential polynomial, then $\underline{d}(P)=\overline{d}(P)$. Thus our Theorem improves, extends, generalizes Theorem E.
\end{rem}

\section{Lemmas} In this section we present some lemmas which will be needed in the sequel. Let $F$, $G$ be two non-constant meromorphic functions. Henceforth we shall denote by $H$ the following function. \be\label{e2.1}H=\left(\frac{\;\;F^{''}}{F^{'}}-\frac{2F^{'}}{F-1}\right)-\left(\frac{\;\;G^{''}}{G^{'}}-\frac{2G^{'}}{G-1}\right).\ee
\begin{lem}\label{l3}\cite{8bc} $N(r,\infty;P) \leq \overline{d}(P)N(r,\infty;f)+\left(\Gamma_P-\overline{d}(P)\right)\overline{N}(r,\infty;f).$
\end{lem}

\begin{lem}\label{l4}\cite{6} Let $f$ be a non-constant meromorphic function and let \[R(f)=\frac{\sum\limits _{i=0}^{n} a_{i}f^{i}}{\sum \limits_{j=0}^{m} b_{j}f^{j}}\] be an irreducible rational function in $f$ with constant coefficients $\{a_{i}\}$ and $\{b_{j}\}$ where $a_{n}\not=0$ and $b_{m}\not=0$. Then $$T(r,R(f))=pT(r,f)+S(r,f),$$ where $p=\max\{n,m\}$.
\end{lem}
\begin{lem} \label{l2.4}\cite{3ab} Let $f$ be a meromorphic function and $P[f]$ be a differential polynomial. Then
$$ m\left(r,\frac{P[f]}{f^{\ol {d}(P)}}\right)\leq (\ol {d}(P)-\underline {d}(P)) m\left(r,\frac{1}{f}\right)+S(r,f).$$
\end{lem}
\begin{lem} \label{l5} \cite{bc1, bc2}Let $f$ be a meromorphic function and $P[f]$ be a differential polynomial. Then we have
\beas N\left(r,\infty;\frac{P[f]}{f^{\ol {d}(P)}}\right)&\leq& (\Gamma _{P}-\ol {d}(P))\;\ol N(r,\infty;f)+(\ol {d}(P)-\underline {d} (P))\; N(r,0;f\mid\geq k+1)\\&&+Q \ol N(r,0;f\mid\geq k+1)+\ol {d}(P) N(r,0;f\mid\leq k)+S(r,f).\eeas
\end{lem}
\begin{lem}\label{l8.5} For the differential polynomial $P[f]$,
 \beas  N(r,0;P[f]) &\leq& (\Gamma-\overline{d}(P))\ol{N}(r,\infty;f)+\underline{d}(P)\ol{N}(r,0;f)\\
 &+& (\overline{d}(P)-\underline{d}(P))\left(m(r,\frac{1}{f})+T(r,f)\right)+S(r,f).\eeas
\end{lem}
\begin{proof} From Lemma \ref{l2.4}, it is clear that
\bea\label{diwali1}\underline{d}(P)m(r,\frac{1}{f})\leq m(r,\frac{1}{P})+S(r,f).\eea
Now using Lemmas \ref{l3}, \ref{l2.4} and (\ref{diwali1}), we have
\beas  N(r,0;P[f]) &=&  T(r,P)-m(r,\frac{1}{P})+O(1)\\
&\leq& T(r,P)-\underline{d}(P)m(r,\frac{1}{f})+S(r,f)\\
&\leq& (\overline{d}(P)-\underline{d}(P))m(r,\frac{1}{f})+\overline{d}(P)m(r,f)\\
&+& \overline{d}(P)N(r,\infty;f)+\left(\Gamma_{P}-\overline{d}(P)\right)\overline{N}(r,\infty;f)\\
&-& \underline{d}(P)m(r,\frac{1}{f})+S(r,f)\\
&\leq& \left(\Gamma_{P}-\overline{d}(P)\right)\overline{N}(r,\infty;f)+\underline{d}(P)\ol{N}(r,0;f)\\
 &+& (\overline{d}(P)-\underline{d}(P))\left(m(r,\frac{1}{f})+T(r,f)\right)+S(r,f).
\eeas
Hence the proof is completed.
\end{proof}
\begin{lem}\label{l8.55} For the differential polynomial $P[f]$,
 \beas  N(r,0;P[f]) &\leq& T(r,P)-\underline{d}(P)T(r,\frac{1}{f})+\underline{d}(P)N(r,\frac{1}{f})+S(r,f).\eeas
\end{lem}
\begin{proof}
Similar to above Lemma.
\end{proof}
\begin{lem}\label{l9.55} Let $j$ and $p$ be two positive integers satisfying $j\geq p+1$ and $ \Gamma> (k+1) \underline{d}(P)- (p+1)$. Then for the differential polynomial $P[f]$,
$$\overline{N}_{_{(j+\Gamma-\underline{d}(P)}}(r,0;f^{\underline{d}(P)}) \leq \overline{N}_{(j}(r,0;P[f]).$$
\end{lem}
\begin{proof} Let $z_{0}$ be a zero of $f$ of order $t$. If $t \underline{d}(P)<j+\Gamma-\underline{d}(P)$, then the proof is obvious. So we assume that $t \underline{d}(P)\geq j+\Gamma-\underline{d}(P)$. Now we consider two cases:\par
\textbf{Case-I} Let us assume that $t\geq k+1$. Then $z_0$ is a zero of $P[f]$ of order atleast
\beas &&\min\limits_{j}\{n_{0j}t+n_{1j}(t-1)+\ldots+n_{kj}(t-k)\}\\
&=&\min\limits_{j}\{t dM_{j}- (\Gamma_{M_{j}}-dM_{j})\}\\
&=&(t+1)\underline{d}(P)-\max\limits_{j}\{\Gamma_{M_{j}}\}\\
&\geq& (j+\Gamma-\underline{d}(P))+\underline{d}(P)-\Gamma\geq j
\eeas
So the proof is clear.\par
\textbf{Case-II} Let us assume that $t\leq k$. Then
\beas k \underline{d}(P) &\geq& t \underline{d}(P) \geq j+\Gamma-\underline{d}(P)\\
&\geq& p+1+\Gamma-\underline{d}(P),
\eeas
which is a contradiction as $ \Gamma> (k+1) \underline{d}(P)- (p+1)$.
\end{proof}
\begin{lem}\label{l9} Let $j$ and $p$ be two positive integer satisfying $j\geq p+1$ and $ \Gamma> (k+1) \underline{d}(P)- (p+1)$. Then for a differential polynomial $P[f]$,
\beas N_{p}(r,0;P[f]) &\leq& N_{_{p+\Gamma-\underline{d}(P)}}(r,0;f^{\underline{d}(P)})+(\Gamma-\underline{d}(P))\overline{N}(r,\infty;f)\\
&&+(\overline{d}(P)-\underline{d}(P))\left(m(r,\frac{1}{f})+T(r,f)\right)+S(r,f).\eeas
\end{lem}
\begin{proof}
From Lemmas \ref{l8.5}, \ref{l9.55}, we have
\beas && N_{p}(r,0;P[f])\\
&\leq& (\Gamma-\overline{d}(P))\ol{N}(r,\infty;f)+\ol{N}(r,0;f^{\underline{d}(P)})+(\overline{d}(P)-\underline{d}(P))(m(r,\frac{1}{f})+T(r,f))\\
&-& \sum\limits_{j=p+1}^{\infty}\ol{N}_{(j}(r,0,P[f])+S(r,f)\\
&\leq& (\Gamma-\overline{d}(P))\ol{N}(r,\infty;f)+N_{_{p+\Gamma-\underline{d}(P)}}(r,0;f^{\underline{d}(P)})+(\overline{d}(P)-\underline{d}(P))(m(r,\frac{1}{f})+T(r,f))\\
&+& \sum\limits_{j=p+\Gamma-\underline{d}(P)+1}^{\infty}\ol{N}_{(j}(r,0;f^{\underline{d}(P)})-\sum\limits_{j=p+1}^{\infty}\ol{N}_{(j}(r,0;P[f])+S(r,f)\\
&\leq& (\Gamma-\overline{d}(P))\ol{N}(r,\infty;f)+N_{_{p+\Gamma-\underline{d}(P)}}(r,0;f^{\underline{d}(P)})+(\overline{d}(P)-\underline{d}(P))(m(r,\frac{1}{f})+T(r,f))+S(r,f).
\eeas
This completes the proof.
\end{proof}
\begin{lem}\label{l9.5} Let $j$ and $p$ be two positive integer satisfying $j\geq p+1$ and $ \Gamma> (k+1) \underline{d}(P)- (p+1)$. Then for a differential polynomial $P[f]$,
\beas N_{p}(r,0;P[f]) &\leq& N_{_{p+\Gamma-\underline{d}(P)}}(r,0;f^{\underline{d}(P)})+T(r,P)-\underline{d}(P)T(r,f)+S(r,f).\eeas
\end{lem}
\begin{proof}
Proof follows from Lemmas \ref{l8.55} and \ref{l9}.
\end{proof}
\begin{lem}\label{l11}(\cite{bc3}) Let $F$ and $G$ share $(1,l)$ and $\overline{N}(r,\infty;F)=\overline{N}(r,\infty;G)$ and $H\not\equiv 0$, where  $F$, $G$ and $H$ are defined as earlier. 
 Then \beas N(r,\infty;H) &\leq& \overline{N}(r,\infty;F)+\overline{N}(r,0;F|\geq 2)+\overline{N}(r,0;G|\geq 2)+\overline{N}_{0}(r,0;F')+\overline{N}_{0}(r,0;G')\\
 &&+\overline{N}_{L}(r,1;F)+\overline{N}_{L}(r,1;G)+S(r,f). \eeas
 \end{lem}
\begin{lem}\label{l12} \cite{bc3} If $F$ and $G$ share $(1,l)$, then
$$\overline{N}_{L}(r,1;F)\leq \frac{1}{2}\overline{N}(r,\infty;F)+\frac{1}{2}\overline{N}(r,0;F)+S(r,F)~~\text{when}~~l\geq 1,$$
  and
$$\overline{N}_{L}(r,1;F)\leq \overline{N}(r,\infty;F)+\overline{N}(r,0;F)+S(r,F)~~\text{when}~~l=0.$$
\end{lem}
\begin{lem}\label{l13}\cite{bc3} Let $F$ and $G$ share $(1,l)$ and $H \not\equiv 0$. Then \beas \overline{N}(r,1;F)+ \overline{N}(r,1;G) &\leq& N(r,\infty;H) + \overline{N}^{(2}_{E}(r,1;F)+\overline{N}_{L}(r,1;F)+\overline{N}_{L}(r,1;G)\\ && +\overline{N}(r,1;G)+S(r,f).\eeas
\end{lem}
\section {Proof of the theorem}
\begin{proof} [Proof of Theorem \ref{t1}]
Let $F=\frac{f}{a}$ and $G=\frac{P[f]}{a}$. Then $F-1=\frac{f-a}{a}$, $G-1=\frac{P[f]-a}{a}$. Since $f$ and $P[f]$ share $(a,l)$, it follows that $F$ and $G$ share $(1,l)$ except the zeros and poles of $a(z)$. Now we consider the following cases.\\
{\bf Case 1} Let $H\not\equiv 0$.\\
\textbf{Subcase-1.1.} Assume $l\geq 1$. Using the Second Fundamental Theorem and Lemmas \ref{l13}, \ref{l11} we get
\bea\nonumber T(r,F)+T(r,G) &\leq& \overline{N}(r,\infty;F)+\overline{N}(r,\infty;G)+\overline{N}(r,0;F)+\overline{N}(r,0;G)+N(r,H) \\
&& \nonumber+ \overline{N}^{(2}_{E}(r,1;F)+\overline{N}_{L}(r,1;F)+\overline{N}_{L}(r,1;G)+\overline{N}(r,1;G)\\
&& \nonumber
-\overline{N}_{0}(r,0;F^{'})-\overline{N}_{0}(r,0;G^{'})+S(r,f)\\
& \label{t2} \leq& 2\overline{N}(r,\infty;F)+\overline{N}(r,\infty;G)+N_{2}(r,0;F)+N_{2}(r,0;G) +\overline{N}^{(2}_{E}(r,1;F)\\
&& \nonumber +2\overline{N}_{L}(r,1;F)+2\overline{N}_{L}(r,1;G)+\overline{N}(r,1;G) +S(r,f).\eea
\textbf{Subsubcase-1.1.1.} Next assume $l\geq 2$. Now by using the inequality (\ref{t2}) and Lemma \ref{l9}, we get
\beas &&T(r,F)+T(r,G)\\
&\leq& 2\overline{N}(r,\infty;F)+\overline{N}(r,\infty;G)+N_{2}(r,0;F)+N_{2}(r,0;G) +\overline{N}^{(2}_{E}(r,1;F)\\
&&+2\overline{N}_{L}(r,1;F)+2\overline{N}_{L}(r,1;G)+\overline{N}(r,1;G)+S(r,f)\\
& \leq& 3\overline{N}(r,\infty;f)+N_{2}(r,0;f)+N_{2}(r,0;G)+ N(r,1;F)+S(r,f)\\
&\leq & 3\overline{N}(r,\infty;f)+N_{2}(r,0;f)+N_{_{2+\Gamma-\underline{d}(P)}}(r,0;f^{\underline{d}(P)})+T(r,P)-\underline{d}(P)T(r,f)\\
&&+T(r,f)-m(r,\frac{1}{f-a})+S(r,f).\eeas
i.e., for any $\varepsilon > 0$
\beas && \underline{d}(P)T(r,f)\\
& \leq & \{4-3\Theta(\infty,f)-\delta_{2}(0,f)+\underline{d}(P)-\underline{d}(P)\delta_{_{2+\Gamma-\underline{d}(P)}}(0;f)+\delta(a,f)+\varepsilon\}T(r,f)+S(r,f).\eeas
i.e.,
\beas 3\Theta(\infty,f)+\delta_{2}(0,f)+\underline{d}(P)\delta_{_{2+\Gamma-\underline{d}(P)}}(0;f)+\delta(a,f)
\leq 4,\eeas which is a contradicts (\ref{be1.1}) of Theorem \ref{t1}.\\
\textbf{Subsubcase-1.1.2.} Next we assume $l=1$. Now inequality (\ref{t2}) and in view of Lemmas \ref{l12}, \ref{l9} and \ref{l9.5}, we get
\beas &&T(r,F)+T(r,G)\\
&\leq& 2\overline{N}(r,\infty;F)+\overline{N}(r,\infty;G)+N_{2}(r,0;F)+N_{2}(r,0;G) +\overline{N}^{(2}_{E}(r,1;F)\\
&& +2\overline{N}_{L}(r,1;F)+2\overline{N}_{L}(r,1;G)+\overline{N}(r,1;G)+S(r,f)\\
& \leq& 2\overline{N}(r,\infty;F)+\frac{3}{2}\overline{N}(r,\infty;G)+\frac{1}{2}\overline{N}(r,0;G)+N_{2}(r,0;f)+N_{2}(r,0;G)\\
&&+\overline{N}^{(2}_{E}(r,1;F)+2\overline{N}_{L}(r,1;F)+\overline{N}_{L}(r,1;G)+\overline{N}(r,1;G)+S(r,f)\\
& \leq& \frac{7}{2}\overline{N}(r,\infty;f)+\frac{1}{2}N_{1}(r,0;G)+N_{2}(r,0;f)+N_{2}(r,0;G)+N(r,1;F)+S(r,f)\\
& \leq& \frac{7+\Gamma-\underline{d}(P)}{2}\overline{N}(r,\infty;f)+\frac{1}{2}N_{_{1+\Gamma-\underline{d}(P)}}(r,0;f^{\underline{d}(P)})+
N_{_{2+\Gamma-\underline{d}(P)}}(r,0;f^{\underline{d}(P)})+N_{2}(r,0;f)\\
&&+\frac{\overline{d}(P)-\underline{d}(P)}{2}(m(r,\frac{1}{f})+T(r,f))+T(r,P)-\underline{d}(P)T(r,f)+T(r,f)-m(r,\frac{1}{f-a})+S(r,f).\eeas
i.e.,
\beas &&(2\underline{d}(P)-\overline{d}(P))T(r,f)\\
& \leq& \frac{7+\Gamma-\underline{d}(P)}{2}\overline{N}(r,\infty;f)+\frac{1}{2}N_{_{1+\Gamma-\underline{d}(P)}}(r,0;f^{\underline{d}(P)})+
N_{_{2+\Gamma-\underline{d}(P)}}(r,0;f^{\underline{d}(P)})\\
&&+N_{2}(r,0;f)-m(r,\frac{1}{f-a})+S(r,f).\eeas
i.e., for any $\varepsilon > 0$
\beas &&(2\underline{d}(P)-\overline{d}(P))T(r,f)\\
& \leq& \{(1+\underline{d}(P)+\frac{7+\Gamma}{2})-\frac{7+\Gamma-\underline{d}(P)}{2}\Theta(\infty,f)-\frac{\underline{d}(P)}{2}\delta_{_{1+\Gamma-\underline{d}(P)}}(r,0;f^{\underline{d}(P)})\\
&&-\underline{d}(P)\delta_{_{2+\Gamma-\underline{d}(P)}}(r,0;f^{\underline{d}(P)})-\delta_{2}(0,f)-\delta(a,f)+\varepsilon\}T(r,f)+S(r,f).\eeas
i.e.,
\beas &&\frac{7+\Gamma-\underline{d}(P)}{2}\Theta(\infty,f)+\frac{\underline{d}(P)}{2}\delta_{_{1+\Gamma-\underline{d}(P)}}(r,0;f^{\underline{d}(P)})+
\underline{d}(P)\delta_{_{2+\Gamma-\underline{d}(P)}}(r,0;f^{\underline{d}(P)})\\
&&+\delta_{2}(0,f)+\delta(a,f)\leq \frac{9+\Gamma}{2}+\overline{d}(P)-\underline{d}(P),
\eeas
which is a contradicts (\ref{be1.2}) of Theorem \ref{t1}.
\\
\textbf{Subcase-1.2.} Assume $l=0$. Then by using the Second Fundamental Theorem and Lemma \ref{l13}, \ref{l11}, \ref{l12}, \ref{l9} and \ref{l9.5}, we get
\bea \nonumber &&T(r,F)+T(r,G)\\
\nonumber&\leq& \overline{N}(r,\infty;F)+\overline{N}(r,0;F)+\overline{N}(r,1;F)+\overline{N}(r,\infty;G) +\overline{N}(r,0;G)+\overline{N}(r,1;G)\\
\nonumber &-&\overline{N}_{0}(r,0;F')-\overline{N}_{0}(r,0;G')+S(r,F)+S(r,G)\\
\nonumber & \leq & \overline{N}(r,\infty;F)+\overline{N}(r,0;F)+\overline{N}(r,\infty;G)+\overline{N}(r,0;G)+N(r,\infty;H)+\overline{N}^{(2}_{E}(r,1;F)\\
\nonumber &+&\overline{N}_{L}(r,1;F)+\overline{N}_{L}(r,1;G)+\overline{N}(r,1;G) -\overline{N}_{0}(r,0;F')-\overline{N}_{0}(r,0;G')+S(r,f)\\
\nonumber & \leq& 2\overline{N}(r,\infty;F)+\overline{N}(r,\infty;G)+N_{2}(r,0;F)+N_{2}(r,0;G)+\overline{N}^{(2}_{E}(r,1;F)\\
\nonumber &+&2\overline{N}_{L}(r,1;F)+2\overline{N}_{L}(r,1;G)+\overline{N}(r,1;G)+S(r,f) \\
\nonumber & \leq & 2\overline{N}(r,\infty;F)+\overline{N}(r,\infty;G)+N_{2}(r,0,f)+N_{2}(r,0;G)+\overline{N}(r,\infty;F)+\overline{N}(r,0;F)\\
\nonumber &+&2(\overline{N}(r,\infty;G)+\overline{N}(r,0;G))+\overline{N}^{(2}_{E}(r,1;F)+\overline{N}_{L}(r,1;F)+\overline{N}(r,1;G)+S(r,f)\\
\nonumber&\leq&  6\overline{N}(r,\infty;f)+N_{2}(r,0,f)+N_{2}(r,0;G)+2\overline{N}(r,0;G)+\overline{N}(r,0;F)+N(r,1;F)+S(r,f)\\
\nonumber&\leq&  6\overline{N}(r,\infty;f)+N_{2}(r,0,f)+\overline{N}(r,0;f)+N_{_{2+\Gamma-\underline{d}(P)}}(r,0;f^{\underline{d}(P)})+T(r,P)-\underline{d}(P)T(r,f)\\
\nonumber&+&2N_{_{1+\Gamma-\underline{d}(P)}}(r,0;f^{\underline{d}(P)})+2(\Gamma-\underline{d}(P))\overline{N}(r,\infty;f)+2(\overline{d}(P)-\underline{d}(P))\left(m(r,\frac{1}{f})+T(r,f)\right)\\
\nonumber&+&T(r,f)-m(r,\frac{1}{f-a})+S(r,f)
\eea
i.e., for any $\varepsilon > 0$
\beas && (5\underline{d}(P)-4\overline{d}(P))T(r,f)\\
&\leq& 2(\Gamma-\underline{d}(P)+3)\overline{N}(r,\infty;f)+\overline{N}(r,0;f)+N_{2}(r,0,f)+N_{_{2+\Gamma-\underline{d}(P)}}(r,0;f^{\underline{d}(P)})\\
\nonumber && +2N_{_{1+\Gamma-\underline{d}(P)}}(r,0;f^{\underline{d}(P)})-m(r,\frac{1}{f-a})+S(r,f)\\
\nonumber&\leq& \{2(\Gamma+4)+\underline{d}(P)-2(\Gamma-\underline{d}(P)+3)\Theta(\infty,f)-\Theta(0,f)-\delta_{2}(0,f)-\delta(a,f)\\
\nonumber&& -\underline{d}(P)\delta_{_{2+\Gamma-\underline{d}(P)}}(r,0;f^{\underline{d}(P)})
-2\underline{d}(P)\delta_{_{1+\Gamma-\underline{d}(P)}}(r,0;f^{\underline{d}(P)})+\varepsilon\}T(r,f)+S(r,f).\eeas
i.e.,
\beas && 2(\Gamma-\underline{d}(P)+3)\Theta(\infty,f)+\Theta(0,f)+\delta_{2}(0,f)+\underline{d}(P)\delta_{_{2+\Gamma-\underline{d}(P)}}(r,0;f^{\underline{d}(P)})\\
&& +2\underline{d}(P)\delta_{_{1+\Gamma-\underline{d}(P)}}(r,0;f^{\underline{d}(P)})+\delta(a,f)\leq 2(\Gamma+4)+4(\overline{d}(P)-\underline{d}(P)),\eeas
which is a contradicts (\ref{be1.3}) of Theorem \ref{t1}.\\
\\
{\bf Case 2} If $H\equiv 0$, then on integration, we get  \be\label{e3.8}\frac{1}{F-1}\equiv\frac{C}{G-1}+D,\ee where $C$, $D$ are constants and $C\not=0$. From (\ref{e3.8}) it is clear that $F$ and $G$ share $1$ CM.
We first assume that $D\not=0$. Then by (\ref{e3.8}) we get \be\label{e3.8a}\ol N(r,\infty;f)=S(r,f).\ee
Now we can write (\ref{e3.8}) as \be\label{e3.9}\frac{1}{F-1}=\frac{D\left(G-1+\frac{C}{D}\right)}{G-1}\ee
Consequently, \be\label{e3.10}\ol N\left(r,1-\frac{C}{D};G\right)=\ol N(r,\infty;F)=\ol N(r,\infty;G)=S(r,f).\ee
\textbf{Subcase-2.1} If $\frac{C}{D}\not=1$, by the second fundamental theorem, Lemma \ref{l9.5}, we have
\beas T(r,G) &\leq& \ol N(r,\infty;G)+ N_{1}(r,0;G)+\ol N\left(r,1-\frac{C}{D};G\right)+S(r,G)\\&\leq&\ol N(r,0;G)+S(r,f)\leq N_{2}(r,0;G)+S(r,f)\\
&\leq&  N_{_{2+\Gamma-\underline{d}(P)}}(r,0;f^{\underline{d}(P)})+T(r,P)-\underline{d}(P)T(r,f)+S(r,f)\eeas
That is, $\delta_{_{1+\Gamma-\underline{d}(P)}}(r,0;f^{\underline{d}(P)})=\delta_{_{2+\Gamma-\underline{d}(P)}}(r,0;f^{\underline{d}(P)})=0$. Also $\Theta(\infty,f)=1$.\par
Now rest part is same as \emph{Subcase 1.1.} of Proof of Theorem 1.3 in \cite{ly}.\par
\textbf{Subcase-2.2} If $\frac{C}{D}=1$, we get from (\ref{e3.8})
\bea\label{e3.12} \left(F-1-\frac{1}{C}\right)G\equiv -\frac{1}{C}.\eea
i.e., \bea\label{e3.12o} \frac{1}{f^{\ol {d}(P)}\left(f-(1+1/C)a\right)}\equiv -\;\frac{C}{a^{2}}\;\;\frac{P[f]}{f^{\ol {d}(P)}}.\eea
From (\ref{e3.12}) it follows that \be \label{e3.13} N(r,0;f\mid \geq k+1)\leq N(r,0;P[f])\leq N(r,0;G)\leq N(r,0;a)=S(r,f).\ee
Applying the first fundamental theorem, (\ref{e3.8a}), (\ref{e3.13}), (\ref{e3.12o}) and Lemmas \ref{l2.4}, \ref{l5}, we get that
\bea\label{e3.14} &&(n+\ol {d}(P))T(r,f)\\
\nonumber&=& T\left(r,\frac{1}{f^{\ol {d}(P)}(f-(1+\frac{1}{C})a)}\right)+S(r,f)\\
\nonumber&\leq& m\left(r,\frac{P[f]}{f^{\ol {d}(P)}}\right)+N\left(r,\frac{P[f]}{f^{\ol {d}(P)}}\right)+S(r,f)
\nonumber\\&\leq& (\ol {d}(P)-\underline {d}(P)) \left[T(r,f)-\{N(r,0;f\mid\leq k)+N(r,0;f\mid \geq k+1)\}\right]+(\ol {d}(P)-\underline {d}(P))\nonumber\\&& N(r,0;f\mid\geq k+1)+\mu \;\ol N(r,0;f\mid\geq k+1)+\ol {d}(P)N(r,0;f\leq k)+S(r,f)\nonumber\\&\leq& (\ol {d}(P)-\underline {d}(P)) T(r,f)+\underline {d}(P) N(r,0;f\mid\leq k)+S(r,f).\nonumber\eea
From (\ref {e3.14}) it follows that \beas nT(r,f)\leq S(r,f),\eeas which is impossible.\\
Hence $D=0$ and so $\frac{G-1}{F-1}=C$ or $\frac{P[f]-a}{f-a}=C$. This proves the theorem. \end{proof}
\begin{center} {\bf Acknowledgement} \end{center}
The authors also wish to thank the referee for his/her valuable remarks and suggestions towards the improvement of the paper.

\end{document}